\documentclass{amsart}
\usepackage{amsmath}
  \usepackage{paralist}
  \usepackage{graphics} 
  \usepackage{epsfig} 
 \usepackage[colorlinks=true]{hyperref}
\hypersetup{urlcolor=blue, citecolor=red}

\newtheorem{theorem}{Theorem}[section]
\newtheorem{corollary}{Corollary}

\newtheorem{lemma}[theorem]{Lemma}
\newtheorem{proposition}{Proposition}

\theoremstyle{definition}
\newtheorem{definition}[theorem]{Definition}
\newtheorem{remark}{Remark}

\title[Periodic Waves for the
KG-NLS System]
      {Stability Properties of Periodic Standing Waves for the
Klein-Gordon-Schr\"odinger System}

\author[F\'abio Natali and Ademir Pastor]{}


 \email{fmanatali@uem.br}
 \email{apastor@impa.br}

\thanks{The second author is supported by CNPq/Brazil under grant no. 152234/2007-1.}

\begin{document}
\maketitle

\centerline{\scshape F\'abio Natali }
\medskip
{\footnotesize
 \centerline{Universidade Estadual de Maring\'a - UEM}
   \centerline{Avenida Colombo, 5790,}
   \centerline{CEP 87020-900,
 Maring\'a, PR, Brazil.}
} 

\medskip

\centerline{\scshape Ademir Pastor}
\medskip
{\footnotesize
 \centerline{ Instituto de Matem\'atica Pura e Aplicada -
IMPA}
   \centerline{Estrada Dona Castorina, 110}
   \centerline{CEP 22460-320, Rio de Janeiro, RJ, Brazil.}
}

\bigskip



\begin{abstract}
We study the existence and orbital stability/instability of periodic
standing wave solutions for the Klein-Gordon-Schr\"odinger system
with Yukawa and cubic interactions. We prove the existence of
periodic waves depending on the Jacobian elliptic functions. For one
hand, the approach used to obtain the stability results is the
classical Grillakis, Shatah and Strauss theory in the periodic
context. On the other hand, to show the instability results we
employ a general criterium introduced by Grillakis, which get
orbital instability from linear instability.
\end{abstract}

\section{Introduction.}

\hspace{0.05cm} In this paper we shall investigate the orbital
stability of periodic standing wave solutions associated to the
Klein-Gordon-Schr\"odinger system (KG-NLS henceforth),

\begin{equation}    \label{KGSCH}
\left\{\begin{array}{lll}
iu_t+\displaystyle\frac{1}{2}\Delta u=-uv\frac{\partial f}{\partial|u|^2}(|u|^2,v)\\\\
v_{tt}-\Delta v+m^2v=f(|u|^2,v)+v\displaystyle\frac{\partial
f}{\partial v}(|u|^2,v),
\end{array}\right.
\end{equation}
when $f(s,t)=s$ and $f(s,t)=st$, here $s,t\in\mathbb{R}$.

If $f(s,t)=s$ equations in $(\ref{KGSCH})$ are the so-called
Klein-Gordon-Schr\"odinger system with Yukawa interaction and it
describes a system of conserved scalar nucleons interacting with
neutral scalar meson. Here
$u:\mathbb{R}^d\times\mathbb{R}\rightarrow\mathbb{C}$ represents a
complex scalar nucleon field and
$v:\mathbb{R}^d\times\mathbb{R}\rightarrow\mathbb{R}$ a real scalar
meson field. The real constant $m^2$ determines the mass of a meson.
The full system \eqref{KGSCH} was motivated  by Hayashi's paper
\cite{hayashi}.

We restrict ourselves to the case $d=1$. The periodic standing waves
we are interested in are of the form
\begin{equation}\label{standingwaves}
u(x,t)=e^{ict}\psi_c(x), \qquad v(x,t)=\phi_c(x),
\end{equation}
where $c \in \mathbb{R}$ and
$\psi_c,\phi_c:\mathbb{R}\rightarrow\mathbb{R}$ are smooth periodic
functions with a fixed period $L>0$.

Studies related to the stability of stationary waves and
well-posedness results for equations $(\ref{KGSCH})$, have been
having a considerable development in recent years. For instance,
when $d=3$, Ohta in \cite{ohta1}, obtained a result of stability for
stationary states for equations $(\ref{KGSCH})$ ($f(s,t)=s$) by
using the variational approach introduced by Cazenave and Lions in
\cite{cazenave1}. Later, Kikuchi and Ohta in \cite{kikuchi-ohta},
established a result of orbital instability related to the same
equation when the wave-speed $c>0$ is sufficiently small. On the
other hand, Baillon and Chadam in \cite{baillon} deduced existence
of global solutions by using the $L^p-L^q$ estimates for the
elementary solutions of the Schr\"odinger equation. Fukuda and
Tsutsumi in \cite{fukuda-tsutsumi}, discussed the initial boundary
value problem of KG-NLS $(\ref{KGSCH})$ and obtained the global
existence of strong solutions in the three-dimensional case, and
later the results were improved in \cite{hayashi1}. Others
contributors can be mentioned as \cite{bachelot}, \cite{colliander},
\cite{efinger}, \cite{hayashi}, \cite{rabsztyn}, \cite{tzirakis}.

 In the one-dimensional case, Tang and Ding in
\cite{tang} (see also \cite{tang1}) studied a result of modulational
instability related to the general Klein-–Gordon–-Schr\"odinger
given by
\begin{equation}\label{gKGSS}
\left\{\begin{array}{llll} iu_t+\alpha u_{xx}+\rho uv+\gamma_1|u|^2u=0\\\\
v_{tt}-c_0^2v_{xx}+m^2v+\gamma_2v^3-\beta|u|^2=0,\end{array}\right.
\end{equation}
where $u$, $v$ and $m^2$ are given as above, $\gamma_i$, $i=1,2$ are
cubic nonlinear auto-interactions, $\beta$ and $\rho$ are quadratic
coupling constants and, $\alpha$ and $c_0$ are constants. Further,
it was found that there are a number of possibilities for the
modulational instability regions due to the generalized dispersion
relation, which relates the frequency and wave-number of the
modulating perturbations. When $\gamma_1=\gamma_2=0$,
$\alpha=\frac{1}{2}$ and $\beta=\rho=c_0^2=1$, equation
$(\ref{gKGSS})$ becomes $(\ref{KGSCH})$ with $f(s,t)=s$.

 In general, the studies about the stability/instability to
Klein-Gordon (KG henceforth) and nonlinear Schr\"odinger (NLS
henceforth) equations have attracted a large set of researchers. It
is known that both KG and NLS equations, specially with cubic
interactions, have wide applications in many physical fields such as
nonlinear optics, nonlinear plasmas, condensed matter and so on.
Besides, a similar system given by equations in $(\ref{gKGSS})$ may
describe the dynamics of coupled electrostatic upper-hybrid and
ion-cyclotron waves in a uniform magnetoplasma (see \cite{tang} and
\cite{yu}).

 In a stability/instability approach if one considers the
general KG equation
\begin{equation} \label{genKG}
u_{tt}-\Delta u+f(|u|^2)u=0,  \qquad t\in \mathbb{R},\,\, x \in
\mathbb{R}^d,
\end{equation}
Grillakis \cite{grillakis} (see also \cite{grillakis1},
\cite{grillakis3}) determined sufficient conditions for the orbital
instability of the standing waves $e^{ict}\varphi(x)$ in the space
of radial functions, where $\varphi(x)=\varphi(|x|)$ has a finite
number of nodes (with some restrictions on the nonlinearity). Others
contributions in this qualitative approach can be mentioned, for
example, \cite{shatah1}, \cite{shatah2} and \cite{shatahstrauss}. In
the periodic context, a recent work due to Natali and Pastor in
\cite{natali-pastor} determined stable/unstable families of periodic
standing wave solutions for equation $(\ref{genKG})$ when $f(v)=1-v$
and $f(v)=v$ (with $d=1$) making use of the abstract theory
established in \cite{grillakis1} and \cite{grillakis2}.

 Next, when the NLS
\begin{equation}\label{schrodinger}
iu_t+\Delta u+|u|^{p-1}u=0,\ \ \ \ t\in\mathbb{R},\
x\in\mathbb{R}^d,
\end{equation}
is considered, a large amount of papers concerning the
stability/instability of standing waves can be found in the current
literature. In particular, Cazenave and Lions in \cite{cazenave1},
determined the existence of stable ones, of the form
$u(x,t)=e^{i\omega t}\varphi(x)$, for equation $(\ref{schrodinger})$
with $d\geq1$ and $1<p<1+4/d$. In \cite{grillakis1} and
\cite{grillakis2} is possible to find a set of sufficient conditions
that determines stability/instability of standing waves for that
equation. In this case, the theory of stablity/instability related
to the periodic case (for $d=1$) have been a terrific development,
for instance \cite{angulo1}, \cite{angulo-linares}, \cite{AN2},
\cite{Pastor}, \cite{Gallay1} and \cite{pastor1} .

 The methods used in the present paper in order to show our
 stability/instability results will be the ones developed by Grillakis \textit{et
al.} \cite{grillakis1}, \cite{grillakis2} and Grillakis
\cite{grillakis}, \cite{grillakis3}. The main reason for this is
because system $(\ref{KGSCH})$, when $f(s,t)=s$ or $f(s,t)=st$, can
be seen as an abstract Hamiltonian system
\begin{equation}\label{hamiltoniano}
\displaystyle\frac{dU(t)}{dt}=J\mathcal{E}'(U(t)),
\end{equation}
where $U=(u_1,v_1,u_2,v_2)=(\mbox{Re}u,v,\mbox{Im}u,v_t)$,
$\mathcal{E}$ represents the energy functional and $J$ is the
skew-symmetric matrix defined by
\begin{equation}    \label{matrixJ}
\displaystyle J=\left(
\begin{array}{cccc}
0 & 0 & 1/2 & 0\\\\
0 & 0 & 0 & 1\\\\
-1/2 & 0 & 0 & 0\\\\
0 & -1 & 0 & 0
\end{array}\right).
\end{equation}

Firstly, we consider the case $f(s,t)=s$. In order  to get explicit
solutions we suppose $\psi_c=\sqrt2 \phi_c$ in
\eqref{standingwaves}. Thus, substituting \eqref{standingwaves} into
\eqref{KGSCH} (with $\psi_c=\sqrt2 \phi_c$), it follows that
$\phi_c$ must satisfy the ordinary differential equation
\begin{equation}\label{standingKGS}
-\phi_c''+2c\phi_c-2\phi_c^2=0,
\end{equation}
where $2c=m^2$. Multiplying equation $(\ref{standingKGS})$ by
$\phi'$ and integrating once, we obtain
\begin{equation}\label{quadra}
[\varphi_{\omega}']^2=\displaystyle\frac{1}{3}[-\varphi_{\omega}^3+
3\omega\varphi_{\omega}+6B_{\varphi_{\omega}}],
\end{equation}
where $\varphi_{\omega}=4\phi_c$, $B_{\varphi_{\omega}}$ is a
nonzero integration constant and $\omega:=\omega(c)=2c$. A positive
solution obtained from $(\ref{quadra})$ depending on the
\textit{cnoidal} Jacobi elliptic function (see Byrd and Friedman
\cite{byrd}) is given by
\begin{equation}  \label{cnoidalsol1}
\displaystyle\varphi_{\omega}(x)=
\beta_2+(\beta_3-\beta_2)cn^2\left(\sqrt{\frac{\beta_3-\beta_1}{12}}x;k\right),
\end{equation}
where $\beta_i$, $i=1,2,3$ and the modulus $k$ depend smoothly on
$\omega$ (and therefore on $c$).

As it is well-known, the main ingredient on the theory developed in
\cite{grillakis1}--\cite{grillakis3} is the spectral properties of
the linear operator arising in the linearized equation around the
traveling wave. Here, by making use of the Floquet theory we can
determine that matrix operators
\begin{equation}
\displaystyle\mathcal{L}_{\mathcal{R},cn}=\left(
\begin{array}{cccc}
\displaystyle-\frac{d^2}{dx^2}+2c-2\phi & -2\sqrt{2}\phi\\\\
-2\sqrt{2}\phi & \displaystyle-\frac{d^2}{dx^2}+2c
\end{array}\right)
\label{matrixopreal}\end{equation} and
\begin{equation}
\displaystyle\mathcal{L}_{\mathcal{I},cn}=\left(
\begin{array}{cccc}
\displaystyle-\frac{d^2}{dx^2}+2c-2\phi & 0\\\\
0 & 1
\end{array}\right)
\label{matrixopimag}
\end{equation}
have the spectral properties required in \cite{grillakis2} and
\cite{grillakis3}. This allow us to prove our stability/instability
results concerning the cnoidal solution in \eqref{cnoidalsol1}.

Secondly, we consider the case $f(s,t)=st$. In this case, if we
substitute \eqref{standingwaves} into \eqref{KGSCH} with
$\psi_c=\phi_c$, we obtain the differential equation (after
integrating once)
\begin{equation}    \label{quadracubic}
[\phi_c']^2=-\phi_c^4+2c\phi_c^2+2B_{\phi_{c}}.
\end{equation}
Thus, a {\it dnoidal} wave solution can be found:
\begin{equation}\label{dnoidalsol}
\phi_c(\xi)=\displaystyle\eta dn\left(\eta\xi;k\right),
\end{equation}
where $\eta>0$ depends smoothly of the wave-speed $c>0$.

 The linear operators arising in this case are:
\begin{equation}
\displaystyle\mathcal{L}_{\mathcal{R},dn}=\left(
\begin{array}{cccc}
\displaystyle-\frac{d^2}{dx^2}+2c-2\phi^2 & -4\phi^2\\\\
-4\phi^2 & \displaystyle-\frac{d^2}{dx^2}+2c-2\phi^2
\end{array}\right)
\label{matrixoprealdn}\end{equation} and
\begin{equation}
\displaystyle\mathcal{L}_{\mathcal{I},dn}=\left(
\begin{array}{cccc}
\displaystyle-\frac{d^2}{dx^2}+2c-2\phi^2 & 0\\\\
0 & 1
\end{array}\right),
\label{matrixopimagdn}
\end{equation}
which also possess the spectral properties needed in
\cite{grillakis2} and \cite{grillakis3} that guarantees our
stability/instability results.

We point out that our orbital stability results will be with respect
to periodic perturbations having the same fundamental period as the
corresponding traveling wave, whereas the instability results will
be with respect to periodic perturbations having twice the
fundamental period as the corresponding traveling wave.

 The question about global well-posedness in the energy space
$H_{per}^1([0,L])\times H_{per}^1([0,L])\times L_{per}^2([0,L])$,
associated to  system $(\ref{KGSCH})$ can be established by a direct
application of Kato's classical theory (see \cite{rabsztyn}). Many
other results of local and global well-posedness for these two
equations can be found in the current literature, as for example
\cite{bachelot}, \cite{bou1}, \cite{colliander}, \cite{tzirakis}.

In order to show the current findings, the paper is organized as
follows. In Section \ref{cnoidalsect} we present an explicit family
of periodic solutions related to equation $(\ref{KGSCH})$ when
$f(s,t)=s$ and study their orbital stability/instability. In Section
\ref{dnoidalsect} we consider the case $f(s,t)=st$ and establish the
existence and stability/instability of another family of periodic
solutions.

\vskip.5cm

\noindent {\bf Notation.}
 For $s\in\mathbb{R}$, the Sobolev space
$H_{per}^{s}([0,L])$ is the set of all periodic distributions such
that
$$
\displaystyle||f||_{H^s{[0,L]}}^2:=||f||_{s}^{2}\equiv
L\sum_{k=-\infty}^{+\infty}(1+|k|^2)^s|\widehat{f}(k)|^2 <\infty,
$$
where $\widehat{f}$ is the Fourier transform of $f$. The symbols
$sn(\cdot;k)$, $dn(\cdot;k)$ and $cn(\cdot;k)$ will denote the
Jacobian elliptic functions of \emph{snoidal}, \emph{dnoidal} and
\emph{cnoidal} type, respectively. Quantities Re($z$) and Im($z$)
denote, respectively, the real and imaginary parts of the complex
number $z$.

 Note that system \eqref{KGSCH}
can be written in three distinct form: besides \eqref{KGSCH}, we can
write it as a first-order system (in $t$) and, finally, separate the
real and imaginary parts of $u$. We use any one of these forms
without further comments.

\section{Orbital stability of cnoidal wave solutions for system
$(\ref{KGSCH})$}    \label{cnoidalsect}

This section is concerned with the existence and orbital
stability/instability of periodic solutions to the KG-NLS system
\begin{equation}\left\{\begin{array}{lll}
iu_t+\displaystyle\frac{1}{2}u_{xx}=-vu\\\\
v_{tt}-v_{xx}+m^2v=|u|^2
\end{array}\right.\label{KG-SCH1}\end{equation} of
the form
\begin{equation}   \label{standing}
(u(x,t),v(x,t))=(e^{ict}\sqrt{2}\phi_c(x),\phi_c(x))
\end{equation}
where $\phi_c:\mathbb{R}\rightarrow\mathbb{R}$ is a smooth positive
periodic function with a fixed period $L>0$, $t\in\mathbb{R}$ and
$c>0$. In fact, for $2c=m^2$, equation $(\ref{KG-SCH1})$ becomes
\begin{equation}
-\phi_c''+2c\phi_c-2\phi_c^2=0. \label{standingKGSCH2}
\end{equation}

\subsection{Existence of standing waves.}

Here our goal consists in showing that equation
$(\ref{standingKGSCH2})$ has a smooth branch, $c\in I\mapsto
\phi_c$, of positive periodic solutions with a fixed period $L>0$
for some parameter interval $I$. First we define
\begin{equation} \label{phivarphi}
\varphi_c:=4\phi_c.
\end{equation}
Letting $2c=\omega$, we obtain from $(\ref{standingKGSCH2})$ and
$(\ref{phivarphi})$ that
\begin{equation}
\displaystyle-\varphi_{\omega}''+\omega\varphi_{\omega}-\frac{1}{2}\varphi_{\omega}^2=0.
\label{EDOvarphi}\end{equation}

The next step is to deduce a $L-$periodic solution
$\varphi=\varphi_{\omega}$ for $(\ref{EDOvarphi})$. Indeed,
multiplying equation $(\ref{EDOvarphi})$ by $\varphi'$ and
integrating once, we get
\begin{equation}
\begin{array}{lll}
[\varphi']^2&=&\displaystyle\frac{1}{3}[-\varphi^3+3\omega\varphi^2
+6B_{\varphi}]\\\\
&=&\displaystyle\frac{1}{3}(\varphi-\beta_1)(\varphi-\beta_2)(\beta_3-\varphi),
\end{array}\label{quadra1KGSCH}
\end{equation}
where $B_{\varphi}$ is a nonzero integration constant and
$\beta_1,\beta_2,\beta_3$ are the zeros of the polynomial
$F(t)=-t^3+3\omega t^2+6B_{\varphi}$. We can suppose, without loss
of generality, that $\beta_3>\beta_2>\beta_1$ and $\beta_3>0$.
Therefore, we should have the relations
\begin{equation}
\displaystyle\left\{\begin{array}{lll}
\beta_1+\beta_2+\beta_3=3\omega,\\\\
\beta_2\beta_1+\beta_3\beta_1+\beta_3\beta_2=0,\\\\
\beta_1\beta_2\beta_3=6B_{\varphi},
\end{array}\right.\label{relroots}
\end{equation}

A periodic solution for the differential equation
$(\ref{EDOvarphi})$ is obtained from the standard {\it direct
integration method} (Byrd and Friedman \cite{byrd}, see also
\cite{angulo2}--\cite{angulo-linares}), namely,
\begin{equation}   \label{cnoidalsol}
\displaystyle\varphi_{\omega}(x)=
\beta_2+(\beta_3-\beta_2)cn^2\left(\sqrt{\frac{\beta_3-\beta_1}{12}}x;k\right),
\qquad k^2=\frac{\beta_3-\beta_2}{\beta_3-\beta_1}.
\end{equation}
Moreover, we conclude from identities $(\ref{relroots})$ and
$(\ref{cnoidalsol})$ that the roots $\beta_1,\beta_2,\beta_3$ must
satisfy $\beta_1<0<\beta_2<\beta_3<3\omega$ and function in
\eqref{cnoidalsol} has fundamental period given by
\begin{equation}\label{relbeta}
\displaystyle
T_{\varphi}=\frac{4\sqrt{3}}{\sqrt{\beta_3-\beta_1}}K(k_{\varphi}),
\end{equation}
where $k=k_{\varphi}$ is the so-called elliptic modulus and $K$
represents the complete elliptic integral of the first kind defined
by
$$
K(k)=\int_0^1\;\frac{dt}{\sqrt{(1-t^2)(1-k^2t^2)}}.
$$
\indent From relations in $(\ref{relroots})$ we obtain, after some
calculations, that
\begin{equation}
\beta_2^2+(\beta_3-3\omega)\beta_2+(\beta_3^2-3\omega\beta_3)=0,
\label{elipse}
\end{equation}
whose value of $\beta_2$ is given by
\begin{equation}\label{beta2}
\beta_2=\displaystyle\frac{1}{2}\left(3\omega-\beta_3+\sqrt{9\omega^2+6\omega\beta_3-3\beta_3^2}\right).
\end{equation}
Therefore, from $(\ref{relroots})$ and $(\ref{beta2})$, we deduce
\begin{equation}\label{beta31}
\beta_3-\beta_1=\displaystyle\frac{1}{2}\left(3\beta_3-3\omega+\sqrt{9\omega^2+6\omega\beta_3-3\beta_3^2}\right)
\end{equation}
and
\begin{equation}\label{beta32}
\beta_3-\beta_2=\displaystyle\frac{1}{2}\left(3\beta_3-3\omega-\sqrt{9\omega^2+6\omega\beta_3-3\beta_3^2}\right).
\end{equation}

Identities $(\ref{beta31})$ and $(\ref{beta32})$ able us to conclude
that the modulus $k$ must satisfy
\begin{equation}\label{modulusk}
k^2=\displaystyle\frac{3\beta_3-3\omega-\sqrt{9\omega^2+6\omega\beta_3-3\beta_3^2}}
{3\beta_3-3\omega+\sqrt{9\omega^2+6\omega\beta_3-3\beta_3^2}}.
\end{equation}
Moreover, thanks to $(\ref{beta32})$ we obtain the inequality
$0<\beta_2<2\omega<\beta_3<3\omega$.

Let $\omega>0$ be fixed, from $(\ref{modulusk})$, asymptotic
properties of $K$ and the fact that $\beta_3 \in (2\omega,3\omega)
\rightarrow T_{\varphi}(\beta_3)$ is a strictly increasing function
(see Theorem $\ref{teocurve}$ below) it follows that $\displaystyle
T_{\varphi}>2\pi/\sqrt{\omega}$. This means that for any $L>0$
fixed, choosing $\omega>0$ such that $\omega>4\pi^2/L^2$ there is a
unique $\beta_3=\beta_3(\omega) \in (2\omega,3\omega)$ such that the
corresponding cnoidal wave give by \eqref{cnoidalsol} has
fundamental period $T_\varphi=L$.

\begin{remark} \label{remark2.1}
The \textit{solitary wave} solution for equation $(\ref{EDOvarphi})$
can be determined from the asymptotic properties of Jacobi elliptic
function $cn$ given in $(\ref{cnoidalsol})$. In fact, for $\omega>0$
fixed, if $\beta_1,\beta_2\rightarrow0$ (then
$\beta_3\rightarrow3\omega$), we get, $k\rightarrow1^-$. On the
other hand, since $cn(\cdot,1^{-})\approx sech(\cdot)$, we obtain
the single-humped function
$$
\varphi_{\omega}(x)=3\omega \,
sech^2\left(\frac{\sqrt{\omega}}{2}x\right),
$$
which is the solitary wave solution for equation
$(\ref{EDOvarphi})$.
\end{remark}

Next, we shall construct a smooth curve, $\omega\mapsto
\varphi_{\omega}$, of cnoidal wave solutions for equation
$(\ref{EDOvarphi})$.

\begin{theorem}   \label{teocurve}
Let $L>0$ be arbitrary but fixed. Consider
$\omega_0>\displaystyle4\pi^2/L^2$ and the unique
$\beta_{3,0}=\beta_3(\omega_0)\in(2\omega_0,3\omega_0)$ such that
$T_{\varphi_{\omega_0}}=L$, then

\begin{enumerate}
  \item there exists an
interval $I(\omega_0)$ around  $\omega_0$, an interval
$B(\beta_{3,0})$ around  $\beta_{3,0}$ and a unique smooth function
$\Gamma:I(\omega_0)\rightarrow B(\beta_{3,0})$ such that
$\Gamma(\omega_0)=\beta_{3,0}$ and
$$
\displaystyle\frac{4\sqrt{6}}{\sqrt{3\beta_3-3\omega+\sqrt{9\omega^2+6\omega\beta_3-3\beta_3^2}}}K(k)=L,
$$
here $\omega\in I(\omega_0)$, $\beta_3=\Gamma(\omega)\in
B(\beta_{3,0})$ and $k^2=k^2(\omega)\in (0,1)$ is given by
$(\ref{modulusk})$.

  \item The cnoidal wave solution in $(\ref{cnoidalsol})$,
$\varphi_{\omega}(\cdot;\beta_1,\beta_2,\beta_3)$, determined by
$\beta_i=\beta_i(\omega)$, $i=1,2,3$, has fundamental period $L$ and
satisfies $(\ref{EDOvarphi})$. Moreover, the mapping
$$\omega\in
I(\omega_0)\mapsto \varphi_{\omega}\in H_{per}^n([0,L]), \qquad
n=0,1,\ldots
$$
is a smooth function.

  \item $I(\omega_0)$ can be chosen as $\displaystyle
\left(4\pi^2/L^2,+\infty\right)$.
\end{enumerate}
 \end{theorem}
\begin{proof}
 The proof is  an application of the Implicit Function Theorem and
 it follows closely the arguments in Angulo and Linares \cite[Theorem
 3.1]{angulo-linares} (see also \cite{angulo3} and \cite{natali}). Thus, we omit it here.
\end{proof}

\begin{remark}
From the Implicit Function Theorem, Theorem $\ref{teocurve}$ and
$(\ref{relbeta})$ we can conclude that function
$\Gamma:I(\omega_0)\rightarrow B(\beta_{3,0})$ given in Theorem
$\ref{teocurve}$ is a strictly increasing function. Moreover,
$\displaystyle dk/d\omega>0$ (this fact can also be found  in
\cite{angulo1} or \cite{Pastor}).
\end{remark}

\begin{remark}\label{obsexist}
From Theorem $\ref{teocurve}$ and identity $(\ref{elipse})$ we are
in a position to conclude that $\displaystyle
\omega=(16K^2\sqrt{1-k^2+k^4})/L^2$,
$\displaystyle\beta_3=16K^2\left[\sqrt{1-k^2+k^4}+1+k^2\right]/L^2$,
$\beta_3-\beta_1=\displaystyle 48K^2/L^2$ and
$\beta_3-\beta_2=\displaystyle 48k^2K^2/L^2$.
\end{remark}

As a consequence of Theorem \ref{teocurve}, we immediately have:

\begin{corollary}   \label{corcurve}
Let $c\in\left(2\pi^2/L^2,+\infty\right)$ and $\omega(c)=2c$. Then
the cnoidal wave $\phi_c=\varphi_{\omega(c)}/4$, where
$\varphi_{\omega(c)}$ is give in Theorem \ref{teocurve}, has
fundamental period $L$ and satisfies $(\ref{standingKGSCH2})$.
Moreover, the mapping
$$
c\in\left(\frac{2\pi^2}{L^2},+\infty\right)\mapsto\phi_c\in
H_{per}^n([0,L]), \qquad n=0,1,\ldots
$$
is a smooth function. In addition, $\displaystyle dk/dc>0$.
\end{corollary}

\subsection{Spectral analysis.}

We start establishing the basic framework for the stability study
introduced in \cite{grillakis2} and \cite{grillakis3}. As we have
already mentioned, we can write system  $(\ref{KG-SCH1})$ as a
Hamiltonian system
\begin{equation}\label{hamilto}
\displaystyle\frac{dU(t)}{dt}=J\mathcal{E}'(U(t)),
\end{equation}
where $U=(u_1,v_1,u_2,v_2)=(\mbox{Re}(u),v,\mbox{Im}(u),v_t)$, $J$
is the matrix defined in \eqref{matrixJ} and $\mathcal{E}$ is the
energy functional
\begin{equation}\begin{array}{lll}   \label{conservada}
\mathcal{E}(U)=\displaystyle\frac{1}{2}\int_{0}^{L}\left[u_{1,x}^2+u_{2,x}^2+v_2^2+v_{1,x}^2+m^2v_1^2
-2v_1(u_1^2+u_2^2)\right]dx.
\end{array}
\end{equation}

We remind the reader that system \eqref{KG-SCH1} also preserves the
$L^2$-norm of $u$, that is, the quantity
\begin{equation}\label{functF}
\mathcal{F}(U)=\displaystyle \int_{0}^{L}(u_1^2+u_2^2)dx
\end{equation}
is a conserved quantity of system \eqref{KG-SCH1}.

In what follow in this section we will denote $\Phi=(\sqrt2 \phi,
\phi,0,0)$, where $\phi=\phi_c$ is the cnoidal wave given in
Corollary \ref{corcurve}.

It is well-known that to apply the abstract  Grillakis \textit{et
al.} theory we need to study the spectral properties of operator
\begin{equation}     \label{Lcn}
\mathcal{L}_{cn}:= \mathcal{E}''(\Phi)+c\mathcal{F}''(\Phi)=
\left(\begin{array}{cccc}
\mathcal{L}_{\mathcal{R},cn} & 0\\\\
0 & \mathcal{L}_{\mathcal{I},cn}
\end{array}\right),
\end{equation}
where
\begin{equation}
\displaystyle\mathcal{L}_{\mathcal{R},cn}=\left(
\begin{array}{cccc}
\displaystyle-\frac{d^2}{dx^2}+2c-2\phi & -2\sqrt{2}\phi\\\\
-2\sqrt{2}\phi & \displaystyle-\frac{d^2}{dx^2}+2c
\end{array}\right)
\label{matrixop1}
\end{equation}
 and
\begin{equation}  \label{matrixop2}
\displaystyle\mathcal{L}_{\mathcal{I},cn}=\left(
\begin{array}{cccc}
\displaystyle-\frac{d^2}{dx^2}+2c-2\phi & 0\\\\
0 & 1
\end{array}\right).
\end{equation}

We begin by studying the spectra of operators
$\mathcal{L}_{\mathcal{R},cn}$ and $\mathcal{L}_{\mathcal{I},cn}$.
More precisely, we have:

\begin{theorem}   \label{teoeigenRe}
Let $\phi=\phi_{c}$ be the \textit{cnoidal} wave solution given by
Corollary \ref{corcurve}. Then
\begin{itemize}
  \item[(i)] operator $\mathcal{L}_{\mathcal{R},cn}$ in \eqref{matrixop1}
  defined in $L_{per}^2([0,L])\times L_{per}^2([0,L])$ whose domain
    is $H_{per}^2([0,L])\times H_{per}^2([0,L])$ has exactly one negative
    eigenvalue which is simple; zero is a simple eigenvalue
whose eigenfunction is $(2\phi'/3, \sqrt2 \phi'/3)$. Moreover, the
remainder of the spectrum is constituted by a discrete set of
eigenvalues.

  \item [(ii)] Operator $\mathcal{L}_{\mathcal{I},cn}$ in \eqref{matrixop2}
  defined in $L_{per}^2([0,L])\times L_{per}^2([0,L])$ whose domain
    is $H_{per}^2([0,L])\times L_{per}^2([0,L])$ has only non-negative
eigenvalues being zero the first one which is simple with
eigenfunction $(\phi,0)$. Moreover, the remainder of the spectrum is
constituted by a discrete set of eigenvalues.
\end{itemize}
\end{theorem}

We point out that from Weyl's essential spectrum theorem, all
operators we study here have only point spectrum.

Before proving Theorem \ref{teoeigenRe}, we note that operator
$\mathcal{L}_{\mathcal{R},cn}$ can be diagonalized under a
similarity transformation. In fact, we consider
$$
A_{\mathcal{R}}=\left(\begin{array}{cccc}
 1 & 1/\sqrt{2}\\\\
-1/\sqrt{2} & 1
\end{array}\right).
$$
Then operator
$\mathcal{L}_{D\mathcal{R}}:=A_{\mathcal{R}}\mathcal{L}_{\mathcal{R},cn}A_{\mathcal{R}}^{-1}$,
is a diagonal operator given by
\begin{equation}
\displaystyle\mathcal{L}_{D\mathcal{R}}=\left(
\begin{array}{cccc}
\mathcal{L}_{1,cn} & 0\\\\
0 & \mathcal{L}_{3,cn}
\end{array}\right),
\label{matrixop4}
\end{equation}
with
\begin{equation}\label{L1cn}
\mathcal{L}_{1,cn}=-\frac{d^2}{dx^2}+2c-4\phi
\end{equation}
and
\begin{equation}\label{L3cn}
\mathcal{L}_{3,cn}=-\frac{d^2}{dx^2}+2c+2\phi.
\end{equation}

The next lemma give us the spectral properties of operators
$\mathcal{L}_{1,cn}$ and $\mathcal{L}_{2,cn}$, where
$\mathcal{L}_{1,cn}$ is defined in \eqref{L1cn} and
\begin{equation}\label{L2cn}
\mathcal{L}_{2,cn}=-\frac{d^2}{dx^2}+2c-2\phi.
\end{equation}

\begin{lemma} \label{lemmaAng}
Let $\phi=\phi_{c}$ be the cnoidal wave given by Corollary
\ref{corcurve}. Then the following spectral properties hold:
\begin{itemize}
    \item [(i)]  operator $\mathcal{L}_{1,cn}$ in $(\ref{L1cn})$ defined
    in $L^2_{per}([0,L])$ with domain $H_{per}^2([0,L])$ has exactly one negative
    eigenvalue which is simple;  zero is an eigenvalue which is
    simple with eigenfunction $\phi'$. Moreover, the remainder
    of the spectrum is constituted by a discrete set of
    eigenvalues.
    \item [(ii)] Operator $\mathcal{L}_{2,cn}$ in $(\ref{L2cn})$
    defined in $L^2_{per}([0,L])$ with domain $H_{per}^2([0,L])$
    has no negative eigenvalues; zero is an eigenvalue, simple
    with eigenfunction $\phi$. Moreover, the remainder
    of the spectrum is constituted by a discrete set of
    eigenvalues.
\end{itemize}
\end{lemma}
\begin{proof}
(i) The main point is that the periodic eigenvalue problem
associated to operator $\mathcal{L}_{1,cn}$,
\begin{equation}\label{op1}
\left\{
\begin{array}{lll}
\mathcal{L}_{1,cn}f=\lambda f \\\\
f(0)=f(L), \quad f'(0)=f'(L),
\end{array}
\right.
\end{equation}
is equivalent (under the transformation $\Psi(x)=f(\theta x),
\theta^2=12/(\beta_3-\beta_1)$) to the periodic eigenvalue problem
\begin{equation}\label{lame1}
\left\{
\begin{array}{lll}
\displaystyle\frac{d^2}{dx^2}\Psi+[\delta-12k^2sn^2(x;k)]\Psi=0\\\\
\Psi(0)=\Psi(2K(k)),\;\;\Psi'(0)=\Psi'(2K(k)),
\end{array}
\right.
\end{equation}
associated to the Lam\'e equation with

\begin{equation}\label{delta}
\delta=\frac{12}{\beta_3-\beta_1}(\lambda+\beta_3-2c)=\frac{12}{\beta_3-\beta_1}(\lambda+\beta_3-\omega).
\end{equation}
The proof of this item follows from Floquet Theory and the basic
ideas can be found in \cite[Theorem 4.1]{angulo-linares} (see also
\cite{angulo2}). However, for the sake of clearness we will list the
basic facts. Indeed, since $\mathcal{L}_{1,cn}\phi'=0$ and $\phi'$
has two zeros in $[0,L)$, then \textit{zero} is the second or the
third eigenvalue of $\mathcal{L}_{1,cn}$. Next, since
$\delta_1=4+4k^2$ is an eigenvalue to $(\ref{lame1})$ with
eigenfunction $\Psi_1(x)=cn(x)sn(x)dn(x)$, we have that
$\lambda_1=0$ is a eigenvalue to $(\ref{op1})$ whose eigenfunction
is $\phi'$. On the other hand, from Ince \cite{In} we have that the
\textit{Lam\'e polynomials},
\begin{equation}\label{pollame}
\begin{array}{lll}
\Psi_0(x)=dn(x)[1-(1+2k^2-\sqrt{1-k^2+4k^4})sn^2(x)],\\\\
\Psi_2(x)=dn(x)[1-(1+2k^2+\sqrt{1-k^2+4k^4})sn^2(x)],
\end{array}\end{equation}
are the eigenfunctions related to other two eigenvalue $\delta_0$,
$\delta_2$ given by,
\begin{equation}\label{eigen1}
\delta_0=2+5k^2-2\sqrt{1-k^2+4k^2},\ \ \
\delta_2=2+5k^2+2\sqrt{1-k^2+4k^2}.
\end{equation}
Since $\Psi_0$ has no zeros in $[0,2K]$ it follows that $\delta_0$
will be the first eigenvalue to $(\ref{lame1})$. Since
$\delta_0<\delta_1$ for all $k\in(0,1)$, we obtain from
$(\ref{delta})$ that $\lambda_0<0$. Therefore $\lambda_0$ is the
first eigenvalue to $\mathcal{L}_{1,cn}$ which is simple. Moreover,
since $\delta_1<\delta_2$ for every $k\in(0,1)$, we obtain from
$(\ref{delta})$ that $\lambda_2>0$. This fact implies that
$\lambda_2$ is the third eigenvalue to $\mathcal{L}_{1,cn}$ and
therefore $\lambda_1=0$ results to be simple.

 (ii) It follows immediately from the Floquet theory since
$\mathcal{L}_{2,cn}\phi=0$ and $\phi$ has no zeros in the interval
$[0,L]$.
\end{proof}

\begin{proof}[Proof of Theorem $\ref{teoeigenRe}$] {\rm {(i)}}
Since $\phi$ is strictly positive, it follows that operator
$\mathcal{L}_{3,cn}$ is strictly positive and
$\sigma(\mathcal{L}_{3,cn})\geq 2c$, where
$\sigma(\mathcal{L}_{3,cn})$ denotes the spectrum of
$\mathcal{L}_{3,cn}$. Next, let
$\overrightarrow{f}=(g,h)^{t}\neq\overrightarrow{0}$ be such that
$\mathcal{L}_{D\mathcal{R}}\overrightarrow{f}=\overrightarrow{0}$,
then $\mathcal{L}_{1,cn}g=0$ and $\mathcal{L}_{3,cn}h=0$. Thus, from
Lemma \ref{lemmaAng} we have $h\equiv0$ and $g=\alpha\phi'$, for
some nonzero real constant $\alpha$. Hence, the kernel of
$\mathcal{L}_{D\mathcal{R}}$ is generated by $(\phi',0)^{t}$. This
implies that the kernel of $\mathcal{L}_{\mathcal{R},cn}$ is
1-dimensional and generated by $(2\phi'/3, \sqrt2 \phi/3)$.

 Next we consider $\lambda<0$ and $\overrightarrow{f}=(g,h)^{t}\neq\overrightarrow{0}$ such that
$\mathcal{L}_{D\mathcal{R}}\overrightarrow{f}=\lambda\overrightarrow{f}$,
then $h\equiv0$ and $\mathcal{L}_{1,cn}g=\lambda g$. Therefore, from
Lemma \ref{lemmaAng}  we must have $\lambda=\lambda_0$ and $g=\beta
\chi_0$, where $\lambda_0$ is the unique negative eigenvalue of
$\mathcal{L}_{1,cn}$ and $\chi_0$ is the corresponding
eigenfunction. This implies part (i) of the lemma.

{\rm {(ii)}} It follows immediately from Lemma \ref{lemmaAng} since
operator $\mathcal{L}_{2,cn}$ has no negative eigenvalues and zero
is a simple eigenvalue.
\end{proof}

Theorem \ref{teoeigenRe} give us the spectral properties needed to
prove our stability results when we consider periodic perturbations
having the same fundamental period as the corresponding wave itself.
However, we are also interested in perturbations having twice the
fundamental period as the corresponding wave. In this regard, the
following lemma is useful.

\begin{lemma}  \label{lemmaAng1}
Let $\phi=\phi_c$ be the cnoidal wave solution  given by Corollary
\ref{corcurve}. Then, the linear operator $\mathcal{L}_{1,cn}$ in
\eqref{L1cn} defined in $L^2_{per}([0,2L])$ with domain
$H_{per}^2([0,2L])$ has its first four eigenvalues simple, being the
eigenvalue zero the fourth one with eigenfunction $\phi'$. Moreover,
if $\chi_1$ and $\chi_2$ denote the eigenfunctions associated to the
second and third eigenvalues then $\chi_i \perp \phi$, $i=1,2$.
\end{lemma}
\begin{proof}
The proof follows the same steps as in Lemma $\ref{lemmaAng}$ and a
more detailed proof can be found in \cite[Theorem
4.2]{angulo-linares}. It is easy to see that
$$
\widetilde{\delta}_0=5+2k^2-2\sqrt{4-k^2+k^4}, \quad
\widetilde{\delta}_1=5+5k^2-2\sqrt{4-7k^2+4k^4}
$$
are the first two eigenvalues for the semi-periodic eigenvalue
problem associated with the Lam\'e equation in \eqref{lame1}. The
eigenfunctions are given, respectively, by
$$
\widetilde{\Psi}_0(x)=cn(x)[1-(2+k^2-\sqrt{4-k^2+k^4})sn^2(x)]
$$
$$
\widetilde{\Psi}_1(x)=sn(x)[3-(2+2k^2-\sqrt{4-7k^2+4k^4})sn^2(x)]
$$
Hence, if $\rho_0$ and $\rho_1$ are the first two eigenvalues for
the semi-periodic eigenvalue problem associated with operator
$\mathcal{L}_{1,cn}$, we obtain from \eqref{delta} (with $\delta$
and $\lambda$ replaced, respectively, with $\widetilde{\delta}$ and
$\rho$) that (see e.g. \cite[Theorem 2.1]{MW})
$$
\lambda_0<\rho_0<\rho_1<\lambda_1=0,
$$
where $\lambda_0$ and $\lambda_1$ are given in Lemma \ref{lemmaAng}.
The orthogonality condition follows from the explicit forms of
$\chi_1$ and $\chi_2$. This completes the proof of the lemma.
 \end{proof}

\subsection{Orbital stability.}   \label{subsection2.3}

First of all, let us make clear our notion of orbital stability.
Since system \eqref{KG-SCH1} has phase and translation symmetries,
that is, if $(u(x,t), v(x,t))$ is a solution of \eqref{KG-SCH1}, so
are
\begin{equation}\label{trans}
(u(x+x_0,t), v(x+x_0,t))
\end{equation}
and
\begin{equation}\label{phase}
(e^{is}u(x,t), v(x,t))=:P_s(u,v)(x,t),
\end{equation}
for any $x_0,s \in \mathbb{R}$, our definition of orbital stability
in this subsection is as follows:

\begin{definition}    \label{orbitalstability}
A standing wave solutions for \eqref{KG-SCH1},
$(e^{ict}\psi_c(x),\phi_c(x))$, is said to be orbitally stable in
$X= H_{per}^1([0,L])\times H_{per}^1([0,L])\times L_{per}^2([0,L])$
if for any $\varepsilon>0$ there exists $\delta>0$ such that if
$(u_0,v_0,v_1)\in X$ satisfies
$||(u_0,v_0,v_1)-(\psi_c,\phi_c,0)||_X<\delta$, then the solution
$\overrightarrow{u}(t)=(u,v,v_t)$ of \eqref{KG-SCH1} with
$\overrightarrow{u}(0)=(u_0,v_0,v_1)$ exists globally and satisfies
$$
\displaystyle\sup_{t\geq0}\inf_{ s,y\in\mathbb{R}}
||\overrightarrow{u}(t)
-(e^{is}T_y\psi_c(\cdot),T_y\phi_c(\cdot),0)||_{X}<\varepsilon.
$$
Otherwise, $(e^{ict}\psi_c(x),\phi_c(x))$ is said to be orbitally
unstable in $X$.
\end{definition}

Here, $T_yg(x)=g(x+y)$. Our stability result reads as follows:

\begin{theorem}   \label{teoestKGSCH}
 Let $\phi_{c}$ be the corresponding cnoidal wave obtained in
 Corollary \ref{corcurve}. Then the periodic wave solution
$(\sqrt{2}e^{ict}\phi_c,\phi_c)$ is orbitally stable in $X$ by the
periodic flow of system \eqref{KG-SCH1}.
\end{theorem}
\begin{proof}
We apply the abstract {\it Stability Theorem} in \cite{grillakis2}.
From Theorem $\ref{teoeigenRe}$ we see that operator
$\mathcal{L}_{cn}$ has the spectral properties required in Grillakis
\textit{et al.} \cite{grillakis2} to apply the abstract theorem.
Indeed,  Theorem $\ref{teoeigenRe}$ implies that operator
$\mathcal{L}_{cn}$ has only one negative eigenvalue which is simple
and its kernel is 2-dimensional. Moreover, the remainder of the its
spectrum consists on a discrete and positive set of eigenvalues.

Also, it is easy to see that our smooth curve of periodic solutions
given in Corollary \ref{corcurve} is a family of critical points for
the functional $\mathcal{H}=\mathcal{E}+c\mathcal{F}$, that is,
\begin{equation}   \label{pointcri}
\mathcal{H}'(\sqrt{2}\phi_{c}, \phi_{c},0,0)=0.
\end{equation}

Finally, we need to study the convexity of the real function
\begin{equation}
d(c)=\mathcal{E}(\sqrt{2}\phi_{c},
\phi_{c},0,0)+c\mathcal{F}(\sqrt{2}\phi_{c},\phi_{c},0,0).
\label{functiond1}
\end{equation}
From \eqref{pointcri} we obtain that
$d'(c)=\mathcal{F}(\sqrt{2}\phi_{c},\phi_{c},0,0)$. Hence,  from
Theorem $\ref{teocurve}$,
\begin{equation}\begin{array}{lll}  \label{productofdc}
d''(c)&=&\displaystyle\frac{d}{dc}\left(\int_{0}^{L}\phi_{c}^2(x)dx\right)\\\\
&=&\displaystyle\frac{1}{8}\frac{d}{d\omega}\left(\int_{0}^{L}\varphi_{\omega}^2(x)dx\right).\\\\
\end{array}
\end{equation}

In order to finish the proof, we just need to show  that $d''(c)>0$.
To show this, we  integrate equation $(\ref{EDOvarphi})$ over
$[0,L]$ to conclude that
$$
\displaystyle\int_{0}^{L}\varphi_{\omega}^2(x)dx=2\omega\int_{0}^{L}\varphi_{\omega}(x)dx.
$$
However, to verify that $d''(c)>0$ it is sufficient  to prove that
$\Upsilon(\omega)=\int_{0}^{L}\varphi_{\omega}(x)dx$ is a strictly
increasing function. In fact, from Byrd and Friedman's book
\cite{byrd} we deduce from the fact
$\beta_3-\beta_2=\displaystyle48k^2K^2/L^2$, that
\begin{equation}
\displaystyle\Upsilon(\omega)=\frac{16}{L}\omega\left(K(k)^2[\sqrt{1-k^2+k^4}+1-2k^2]
+3K(k)[E(k)-k'^2K(k)]\right). \label{intvarphi}\end{equation}

Since
\begin{equation}\label{dk}\begin{array}{llll}
\displaystyle f(k)&=&\displaystyle K(k)^2[\sqrt{1-k^2+k^4}+1-2k^2]
+3K(k)[E(k)-k'^2K(k)]\\\\
&=&K(k)^2[\sqrt{1-k^2+k^4}+k^2-2]
+3K(k)E(k),\\\\
\end{array}
\end{equation}
function $f:(0,1)\rightarrow \mathbb{R}$ defined in $(\ref{dk})$ is
a strictly increasing function with respect to the modulus.
Moreover, from Remark $\ref{obsexist}$ we have $\displaystyle
dk/d\omega>0$. Thus, we conclude that $\Upsilon(\omega)$ is a
strictly increasing function and therefore $d''(c)>0$.
\end{proof}

\begin{remark}\label{obs1}
By using the same framework as presented in this section (with the
necessary modifications), one shows the stability of the solitary
standing wave solution $(\sqrt{2}e^{ict}\phi_c(x),\phi_c(x))$,
$x\in\mathbb{R}$, $c>0$, for equation $(\ref{KGSCH})$ with
$f(s,t)=s$, where
$$\phi_c(x)=\displaystyle 6c\, sech^2\left(\sqrt{\frac{c}{2}}x\right),$$
is the solitary wave given in Remark \ref{remark2.1}.
\end{remark}

\subsection{Orbital instability.} \label{subsection2.4}

Let $\phi=\phi_c$ be given in Corollary \ref{corcurve}. This section
is devoted to prove that the corresponding standing wave is unstable
in the space $H_{per}^1([0,2L])\times H_{per}^1([0,2L])\times
L_{per}^2([0,2L])$, that is, the standing waves are unstable with
respect to periodic perturbations having twice the fundamental
period as the corresponding wave. The ideas to obtain such result is
to get orbital instability from the linear instability of the zero
solution for the linearization of \eqref{KG-SCH1} around the orbit
generated modulus phase, $\{P_{ct}(\sqrt2\phi,\phi,0,0);\,\, t \in
\mathbb{R}\}$, where $P_s$ is the transformation defined in
\eqref{phase}. Note that $P_s$ (acting on real-valued functions) can
be described as
\begin{equation}    \label{matrixpahse}
\displaystyle P_s\left(
\begin{array}{c}
u_1\\
v_1\\
u_2\\
v_2
\end{array}\right)
=\left(
\begin{array}{cccc}
\cos s & 0 & -\sin s & 0\\
0 & 1 & 0 & 0\\
\sin s & 0 & \cos s & 0\\
0 & 0 & 0 & 1
\end{array}\right)
\left(
\begin{array}{c}
u_1\\
v_1\\
u_2\\
v_2
\end{array}\right).
\end{equation}

The clever reader has already noted that our notion of stability
here is slightly different from that one in Subsection 2.3. In fact,
our definition is only modulo phase.

\begin{definition} \label{definition6.1}
Let $Y=H_{per}^1([0,2L])\times H_{per}^1([0,2L]) \times
L_{per}^2([0,2L])$. The orbit generated modulus phase,
$\{P_{sc}(\psi_c,\phi_c,0); s \in {\mathbb{R}} \}$, is said to be
orbitally stable in $Y$, if for every $\varepsilon>0$ there exists a
$\delta>0$ such that if $z_0 \in Y$ and $\|z_0-(\psi_c, \phi_c,0)
\|_{Y}< \delta$, then the solution $z(t)$ of \eqref{KG-SCH1} with
$z(0)=z_0$ exists for all $t$ and
$$
{\displaystyle \sup_{t\geq 0}} \inf_{s \in {\mathbb{R}}} \|
z(t)-P_s(\psi_c,\phi_c,0) \|_{Y} < \varepsilon.
$$
Otherwise, the orbit is said to be orbitally unstable in $Y$.
\end{definition}

First of all we note that linearizing system \eqref{KG-SCH1} around
the orbit $$\Omega=\{ P_{cs}(\sqrt2\phi,\phi,0,0), \,\, s \in
\mathbb{R} \},$$ we get the equation
\begin{equation}\label{111}
\frac{dV}{dt}=J\mathcal{L}_{cn}V+O(\|V\|^2),
\end{equation}
where $V(t)=P_{(-ct)}U(t)-(\sqrt2\phi,\phi,0,0)$, $J$ and
$\mathcal{L}_{cn}$ are the operators defined in \eqref{matrixJ} and
\eqref{Lcn}, respectively.

In order to prove that  \eqref{111} has the zero solution as an
unstable solution we know that it is sufficient to prove that
$J\mathcal{L}_{cn}$ has  finitely many eigenvalues with strictly
positive real part. Moreover, this implies that orbit $\Omega$ is
orbitally unstable (see \cite{grillakis2}, \cite{grillakis}).
Actually, we prove:

\begin{theorem}  \label{theoreminst}
Let  $\phi=\phi_{c}$ be the cnoidal wave given by Corollary
$\ref{corcurve}$. Then the orbit
$$
\Omega=\{ P_{cs}(\sqrt2\phi,\phi,0,0), \,\, s \in \mathbb{R} \}
$$
is orbitally unstable in the space $Y=H_{per}^1([0,2L])\times
H_{per}^1([0,2L]) \times L_{per}^2([0,2L])$.
\end{theorem}
\begin{proof}
From Lemma 5.6 and Theorem 5.8 in \cite{grillakis2}, we know that
$J\mathcal{L}_{cn}$ has finitely many eigenvalues with strictly
positive real part. Hence we only have to prove that
$J\mathcal{L}_{cn}$ has at least one eigenvalue with strictly
positive real part. To prove this, we use the approach introduced by
Grillakis \cite{grillakis3}. We start by defining
$$
\mathcal{Z}= [Ker(\mathcal{L}_{\mathcal{R},cn})\cup
Ker(\mathcal{L}_{\mathcal{I},cn})]^{\bot},
$$
$$
\widehat{\mathcal{L}}_{\mathcal{R},cn}= \mbox{restriction of}\,\,
\mathcal{L}_{\mathcal{R},cn}\,\, \mbox{on} \,\, \mathcal{Z} \cap
H_{per}^2([0,2L]) ,
$$
$$
\widehat{\mathcal{L}}_{\mathcal{I},cn}^{-1}= \mbox{restriction
of}\,\, \mathcal{L}_{\mathcal{I},cn}^{-1}\,\, \mbox{on } \,\,
\mathcal{Z} \cap H_{per}^2([0,2L]).
$$
With this definitions, Theorem 2.6 in \cite{grillakis3} states that
$J\mathcal{L}_{cn}$ has exactly
\begin{equation}   \label{112}
\max\{ n(\widehat{\mathcal{L}}_{\mathcal{R},cn}),
n(\widehat{\mathcal{L}}_{\mathcal{I},cn}^{-1}) \} -
d(C(\widehat{\mathcal{L}}_{\mathcal{R},cn}) \cap
C(\widehat{\mathcal{L}}_{\mathcal{I},cn}^{-1}))
\end{equation}
$\pm$ pairs of real eigenvalues, where $C(\mathcal{L})=\{ z \in
\mathcal{Z}; \,\, \langle \mathcal{L}z, z\rangle_{L_{per}^2} <0 \}$
denotes the negative cone of operator $\mathcal{L}$ and $
d(C(\mathcal{L}))$ denotes the dimension of a maximal linear
subspace that is contained in $C(\mathcal{L})$.

Hence we just need to prove that the number in \eqref{112} is
strictly positive. We first observe that since zero is the first
eigenvalue of $\mathcal{L}_{\mathcal{I},cn}$ (see Theorem
\ref{teoeigenRe}) it follows that such operator is a positive
operator on $\mathcal{Z}$ and so
$C(\widehat{\mathcal{L}}_{\mathcal{I},cn}^{-1}) =\emptyset$ and
$n(\widehat{\mathcal{L}}_{\mathcal{I},cn}^{-1})=0$. Thus, the number
in \eqref{112} reduces to
$n(\widehat{\mathcal{L}}_{\mathcal{R},cn})$.

Let us prove that $n(\widehat{\mathcal{L}}_{\mathcal{R},cn})=2$.
Indeed, from the definition of $\mathcal{L}_{\mathcal{R},cn}$ and
Lemma \ref{lemmaAng1} we see that
$n(\widehat{\mathcal{L}}_{\mathcal{R},cn})\leq
n(\mathcal{L}_{\mathcal{R},cn})\leq 3$. Now let $\lambda_0,
\lambda_1$ and $\lambda_2$ be the three negative eigenvalues of
operator $\mathcal{L}_{1,cn}$ given by Lemma \ref{lemmaAng1} with
respective eigenfunctions $\chi_0, \chi_1$ and $\chi_2$. Thus,
$$
\overrightarrow{\chi}_0=(2\chi_0/3, \sqrt2 \chi_0/3),  \qquad
\overrightarrow{\chi}_1=(2\chi_1/3, \sqrt2 \chi_1/3)  \quad
\mbox{and} \quad \overrightarrow{\chi}_2=(2\chi_2/3, \sqrt2
\chi_2/3)
$$
are eigenfunctions of $\mathcal{L}_{\mathcal{R},cn}$ associated to
eigenvalues $\lambda_0, \lambda_1$ and $\lambda_2$, respectively.
Since $Ker(\mathcal{L}_{\mathcal{I},cn})$ is generated by
$\overrightarrow{\Phi}=(\phi,0)$ (see Theorem \ref{teoeigenRe}), we
obtain from Lemma \ref{lemmaAng1} that
$$
\langle \overrightarrow{\chi}_j, \overrightarrow{\Phi}
\rangle_{L_{per}^2\times L_{per}^2} = \frac{2}{3} \langle \chi_j,
\phi \rangle_{L_{per}^2} = 0, \qquad j=1,2.
$$
Moreover, $\langle \overrightarrow{\chi}_0, \overrightarrow{\Phi}
\rangle_{L_{per}^2\times L_{per}^2} = 2/3 \langle \chi_0, \phi
\rangle_{L_{per}^2} > 0$ since $\chi_0$ and $\phi$ are strictly
positive functions. This completes the proof of the Theorem.
\end{proof}

\section{Orbital stability of dnoidal wave solutions for system
$(\ref{KGSCH})$.}  \label{dnoidalsect}

The purpose of this section is to establish the
stability/instability of dnoidal wave solutions for system
$(\ref{KGSCH})$ when $f(s,t)=st$, namely,
\begin{equation}  \label{KG-SCH2}
\left\{\begin{array}{lll}
iu_t+\displaystyle\frac{1}{2}u_{xx}=-v^2u\\\\
v_{tt}-v_{xx}+m^2v=2|u|^2v.
\end{array}\right.
\end{equation}
Here we look for solutions of the form
$$
u(x,t)=e^{ict}\phi_c(x), \qquad v(x,t)=\phi_c(x)
$$
where $\phi_c$ is a smooth $L-$periodic real-valued function and
$c>0$ is the wave-speed. Substituting this form in $(\ref{KG-SCH2})$
we see, after integration, that $\phi_c$ must satisfy
\begin{equation}\label{quadra2}
[\phi_c']^2=-\phi_c^4+2c\phi_c^2+2B_{\phi_c},
\end{equation}
for some real constant $B_{\phi_c}$.  A smooth positive  periodic
solution for \eqref{quadra2} is given by (see e.g., \cite{angulo1}
or \cite{byrd})
\begin{equation}\label{dnsol}
\phi_c(x)=\eta\, dn(\eta x;k), \qquad k\in(0,1),
\end{equation}
where
\begin{equation}  \label{vel}
\displaystyle \eta=\displaystyle\frac{2K}{L} \qquad  \mbox{and}
\qquad c=\frac{2K^2}{L^2}(2-k^2).
\end{equation}

\begin{remark}\label{obs2}
The solitary standing wave solution related to equation
$(\ref{quadra2})$ with $B_{\phi_c}\equiv0$ can be obtained from
equation $(\ref{dnsol})$, namely,
\begin{equation}\label{sechsol}
\phi_c(x)=\sqrt{c}\, sech(\sqrt{c}x),\ \ \ \ \ x\in\mathbb{R}.
\end{equation}
\end{remark}

 If we consider the same steps as that ones for proving Corollary
 $\ref{corcurve}$, for each $L>0$ we can construct  a smooth branch
 (depending on $c$) of dnoidal waves having fundamental period $L$.
 The next proposition summarizes these results.

 \begin{proposition}  \label{propcurve}
Let $L>0$ fixed  and
$c\in\displaystyle\left(\pi^2/L^2,+\infty\right)$. Then the dnoidal
wave $\phi_c$ given by \eqref{dnsol} has fundamental period $L$ and
satisfies \eqref{quadra2}. Moreover, the mapping
$$
c\in\displaystyle\left(\pi^2/L^2,+\infty\right)\rightarrow\phi_{c}\in
H_{per}^n([0,L])
$$
is a smooth function and the modulus $k=k(c)$ satisfies $dk/dc>0$.
 \end{proposition}

\begin{remark}
The fact that $dk/dc>0$ follows immediately from \eqref{vel} and the
Inverse Function Theorem.
\end{remark}

\subsection{Spectral analysis.}

Let $\phi=\phi_c$ be the dnoidal wave given by Proposition
\ref{propcurve}. As in Subsection 2.2, we first note that system
\eqref{KG-SCH2} can be write as an infinite-dimensional Hamiltonian
system, namely,
$$
\frac{dU}{dt}=J\mathcal{E}_1'(U(t)),
$$
where  $U=(u_1,v_1,u_2,v_2)=(\mbox{Re}(u),v,\mbox{Im}(u),v_t)$, $J$
is the skew-symmetric  matrix defined in \eqref{matrixJ} and
$\mathcal{E}_1$ is the energy functional
\begin{equation}  \label{conservada1}
\begin{array}{lll}
\mathcal{E}_1(U)=\displaystyle\frac{1}{2}\int_{0}^{L}\left[u_{1,x}^2+u_{2,x}^2+v_2^2+v_{1,x}^2+m^2v_1^2
-2v_1^2(u_1^2+u_2^2)\right]dx.
\end{array}
\end{equation}

We observe that functional $\mathcal{F}$ given in \eqref{functF} is
also a conserved quantity of system \eqref{KG-SCH2}. Hence, the
linearized operator is this case is given by
\begin{equation}\label{operatorprincdn}
\mathcal{L}_{dn}:=
\mathcal{E}_1''(\phi,\phi,0,0)+c\mathcal{F}''(\phi,\phi,0,0)=
\left(\begin{array}{cccc}
\mathcal{L}_{\mathcal{R},dn} & 0\\\\
0 & \mathcal{L}_{\mathcal{I},dn}\end{array}\right),
\end{equation}
where
\begin{equation}  \label{matrixoprealdn1}
\displaystyle \mathcal{L}_{\mathcal{R},dn}=\left(
\begin{array}{cccc}
\displaystyle-\frac{d^2}{dx^2}+2c-2\phi^2 & -4\phi^2\\\\
-4\phi^2 & \displaystyle-\frac{d^2}{dx^2}+2c-2\phi^2
\end{array}\right)
\end{equation}
 and
\begin{equation} \label{matrixopimagdn1}
\displaystyle\mathcal{L}_{\mathcal{I},dn}=\left(
\begin{array}{cccc}
\displaystyle-\frac{d^2}{dx^2}+2c-2\phi^2 & 0\\\\
0 & 1
\end{array}\right).
\end{equation}

As in Section \ref{cnoidalsect}, we see that operator
$\mathcal{L}_{\mathcal{R},dn}$ also can be diagonalized by a
similarity transformation. Indeed, let
$$
B_{\mathcal{R}}=\left(\begin{array}{cccc}
 1 & 1\\\\
-1 & 1
\end{array}\right).
$$
Then the operator
$B_{\mathcal{R}}\mathcal{L}_{\mathcal{R},dn}B_{\mathcal{R}}^{-1}=:
\mathcal{L}_{D\mathcal{R}}$ is a diagonal operator given by
\begin{equation}  \label{matrixopdn}
\displaystyle\mathcal{L}_{D\mathcal{R}}=\left(
\begin{array}{cccc}
\displaystyle \mathcal{L}_{1,dn} & 0\\\\
0 &  \mathcal{L}_{3,dn}
\end{array}\right),
\end{equation}
where
\begin{equation}\label{L1dn}
\mathcal{L}_{1,dn}= -\frac{d^2}{dx^2}+2c-6\phi^2
\end{equation}
and
\begin{equation}\label{L3dn}
\mathcal{L}_{3,dn}= \displaystyle-\frac{d^2}{dx^2}+2c+2\phi^2
\end{equation}

Now we can prove:

\begin{theorem}   \label{teoeigenRedn}
Let $\phi=\phi_{c}$ be the \textit{dnoidal} wave solution given by
Proposition  \ref{propcurve}. Then,
\begin{itemize}
  \item[(i)] operator $\mathcal{L}_{\mathcal{R},dn}$ in \eqref{matrixoprealdn1}
  defined in $L_{per}^2([0,L])\times L_{per}^2([0,L])$ whose domain
    is $H_{per}^2([0,L])\times H_{per}^2([0,L])$ has exactly one negative
    eigenvalue which is simple; zero is a simple eigenvalue.
Moreover, the remainder of the spectrum is constituted by a discrete
set of eigenvalues.

  \item [(ii)] Operator $\mathcal{L}_{\mathcal{I},dn}$ in \eqref{matrixopimagdn1}
  defined in $L_{per}^2([0,L])\times L_{per}^2([0,L])$ whose domain
  is
 $H_{per}^2([0,L])\times L_{per}^2([0,L])$ has only non-negative
eigenvalues being zero the first one which is simple. Moreover, the
remainder of the spectrum is constituted by a discrete set of
eigenvalues.
\end{itemize}
\end{theorem}
\begin{proof}
 The procedure here follows from the same arguments in
Theorem $\ref{teoeigenRe}$ and therefore we will give only the main
steps.

To prove (i), we observe that operator $\mathcal{L}_{1,dn}$ in
\eqref{L1dn}, defined in $L_{per}^2([0,L])$ with domain
$H_{per}^2([0,L])$ has exactly one negative eigenvalue and zero is a
simple eigenvalue with eigenfunction $\phi'$. Moreover, operator
$\mathcal{L}_{3,dn}$ in \eqref{L3dn} (defined in $L_{per}^2([0,L])$
with domain $H_{per}^2([0,L])$) is such that
$\sigma(\mathcal{L}_{3,dn}) \geq 2c$. Hence from \eqref{matrixopdn}
and arguments similar as the ones in Theorem \ref{teoeigenRe} we
prove part (i).

To prove part (ii), we just note that zero is the first eigenvalue
of operator $\mathcal{L}_{2,dn}=-\frac{d^2}{dx^2}+2c-2\phi^2$ with
eigenfunction $\phi$. So the proof is completed.
\end{proof}

\begin{remark}
To show the spectral properties of operator $\mathcal{L}_{1,dn}$
used in the proof of Theorem \ref{teoeigenRedn} one needs to see
that the periodic eigenvalue problem associated to
$\mathcal{L}_{1,dn}$, posed on the interval $[0,L]$, is equivalent
to the periodic eigenvalue problem associated to the Lam\'e operator
\begin{equation}  \label{lameop}
\mathcal{L}_{Lame}=-\frac{d^2}{dx^2}+6k^2sn^2(x;k),
\end{equation}
posed on the interval $[0,2K]$ (see \cite{angulo1}, \cite{Pastor}).
\end{remark}

Next lemma is in the same spirit of  Lemma \ref{lemmaAng1}.

\begin{lemma}  \label{lemma2L}
Let $\phi=\phi_c$ be the cnoidal wave solution  given by Proposition
\ref{propcurve}. Then, linear operator $\mathcal{L}_{1,dn}$ in
\eqref{L1dn} defined in $L^2_{per}([0,2L])$ with domain
$H_{per}^2([0,2L])$ has its first four eigenvalues simple, being the
eigenvalue zero the fourth one with eigenfunction $\phi'$. Moreover,
if $\xi_1$ and $\xi_2$ denote the eigenfunctions associated to the
second and third eigenvalues then $\xi_i \perp \phi$, $i=1,2$.
\end{lemma}
\begin{proof}
The proof follows combining the periodic and semi-periodic
eigenvalues problems associated to  operator $\mathcal{L}_{1,dn}$
with the equivalent  problem associated to the Lam\'e operator in
\eqref{lameop} (see e.g., \cite{Pastor}).
\end{proof}

\subsection{Stability results.}
 The procedure here follows the same steps as in Subsection \ref{subsection2.3}.
Our stability theorem reads as follows:

\begin{theorem}\label{teoestKGSCH1}
 Let $\phi_{c}$ be the dnoidal wave solution given by Proposition
\ref{propcurve}. Then the periodic wave solution
$(e^{ict}\phi_c,\phi_c)$ is orbitally stable in $X$ by the periodic
flow of system \eqref{KG-SCH2} in the sense of Definition
\ref{orbitalstability}.
\end{theorem}
\begin{proof}
We apply the abstract {\it Stability Theorem} in \cite{grillakis2}
again. It follows from Theorem \ref{teoeigenRedn} that operator
$\mathcal{L}_{dn}$ has a unique negative eigenvalue, its kernel is
two-dimensional and the remainder of the spectrum is bounded away
from zero. Moreover, by our construction
\begin{equation}\label{crit1}
\mathcal{E}_1'(\phi_c,\phi_c,0,0)+c\mathcal{F}'(\phi_c,\phi_c,0,0)=0.
\end{equation}
Hence, it remains only to show that $d''(c)>0$ where
$$
d(c)=\mathcal{E}(\phi_{c},
\phi_{c},0,0)+c\mathcal{F}(\phi_{c},\phi_{c},0,0).
$$
From \eqref{crit1} we have
$d'(c)=\mathcal{F}(\phi_{c},\phi_{c},0,0)$. Therefore,
\begin{equation}\begin{array}{lll}
d''(c)&=&\displaystyle\frac{d}{dc}\left(\int_{0}^{L}\phi_{c}^2(x)dx\right)\\\\
&=&\displaystyle\frac{4}{L}\frac{d}{dc}\left(K(k)\int_{0}^{K}dn^2(x)dx\right)
=\frac{4}{L}\frac{d}{dk}(K(k)E(k))\frac{dk}{dc}.
\end{array}\label{productofdc1}\end{equation}
Since $k\in (0,1)\mapsto K(k)E(k)$ is a strictly increasing function
and $\displaystyle dk/dc>0$ (see Proposition \ref{propcurve}) we see
from \eqref{productofdc1} that $d''(c)>0$. This completes the proof
of the theorem.
\end{proof}

\begin{remark}
As mentioned in Remark $\ref{obs1}$,  by similar arguments presented
in this section, one can obtain the stability of the solitary
standing wave solution of the form $(e^{ict}\phi_c(x),\phi_c(x))$,
$x\in\mathbb{R}$, $c>0$, for equation $(\ref{KGSCH})$ with
$f(s,t)=st$, where $\phi_c$ is given by $(\ref{sechsol})$.

\end{remark}

\subsection{Instability results}

The idea here is to prove a similar result as the one in Subsection
\ref{subsection2.4}, in which we prove an orbital instability result
taking the advantage that the linearized system has the zero
solution as an unstable solution.

We first observe that linearizing system \eqref{KG-SCH2} around the
orbit, $$\Gamma=\{ P_{ct}(\phi_c,\phi_c,0,0); \,\, t \in
\mathbb{R}\},$$ where $P_s$ is defined in \eqref{phase}, we face the
equation
\begin{equation}\label{222}
\frac{dW}{dt}=J\mathcal{L}_{dn}W+O(\|W\|^2),
\end{equation}
where $W(t)=P_{(-ct)}U(t)-(\phi_c,\phi_c,0,0)$, $J$ and
$\mathcal{L}_{dn}$ are defined in \eqref{matrixJ} and
\eqref{operatorprincdn}, respectively.

Thus we can prove:

\begin{theorem}  \label{theoreminst1}
Let  $\phi_c$ be the dnoidal wave given by Proposition
$\ref{propcurve}$. Then the orbit
$$
\Gamma=\{ P_{cs}(\phi_c,\phi_c,0,0), \,\, s \in \mathbb{R} \}
$$
is orbitally unstable in  $Y$ in the sense of Definition
\ref{definition6.1}.
\end{theorem}
\begin{proof}
The proof follows the same analytic-functional approach introduced
in Theorem \ref{theoreminst} (with obvious modifications), taking
into account Theorem \ref{teoeigenRedn} and Lemma \ref{lemma2L}
instead of Theorem \ref{teoeigenRe} and Lemma \ref{lemmaAng1}. So,
we will omit the details.
\end{proof}

\textbf{Acknowledgement:} The authors would like to thank J. Angulo
for the interest and comments with regard to this work and the
anonymous referee for giving constructive suggestions which allow to
improve the present manuscript.

\end{document}